\documentclass[]{theclass}
\usepackage{mathtools}
\definecolor{orchid}{rgb}{0.85, 0.44, 0.84}
\begin{document}

\begin{frontmatter}

\titledata{The Pairing-Hamiltonian property in Cartesian products of graphs}{}

\authordata{Federico Romaniello}
{Dipartimento per l'Innovazione Umanistica, Scientifica e Sociale\\ Universit\`{a} degli Studi della Basilicata, Potenza, Italy\\ }{federico.romaniello@unibas.it}{}


\begin{abstract}
Let $G$ be a simple graph of even order at least four, and let
$K_G$ denote the complete graph on $V(G)$. A perfect matching of $K_G$
is called a pairing of $G$. The graph $G$ has the Pairing-Hamiltonian
property, or PH-property, if every pairing $M$ of $G$ admits a perfect
matching $N\subseteq E(G)$, disjoint from $M$, such that $M\cup N$ is a
Hamiltonian cycle of $K_G$.

We prove that the PH-property is preserved under Cartesian products.
More precisely, for graphs $G$ and $H$ of even order at least four,
we show that both $G$ and $H$ are PH if and only if every pairing of $G\square H$
admits a Hamiltonian completion contained in a spanning union of
vertex-disjoint prisms determined by a perfect matching of $G$ or
of $H$. Without this support restriction, the converse fails:
a Cartesian product may be PH even when neither of the two graphs
is PH.
\end{abstract}
\end{frontmatter}

\section{Introduction}\label{sec:intro}
The study of extending
perfect matchings to Hamiltonian cycles goes back to the work of Las Vergnas
\cite{LasVergnas1972} and H\"aggkvist \cite{Haggkvist1979} in the 1970s, who
established sufficient conditions of Ore type for such extensions.

An important line of research concerns hypercubes. In 1996, Kreweras
\cite{Kreweras1996} conjectured that every perfect matching of the hypercube
$Q_d$, for $d\geq2$, can be extended to a Hamiltonian cycle.
Fink \cite{Fink2007} proved this conjecture in 2007 by establishing a stronger
statement: every perfect matching of the complete graph on $V(Q_d)$ can be
completed to a Hamiltonian cycle by adding only edges of $Q_d$.

To formulate this property for general graphs, let $K_G$ denote the complete
graph on the same vertex set as a graph $G$. A perfect matching of $K_G$ is
called a \emph{pairing} of $G$. A graph $G$ of even order at least four has the
\emph{Pairing-Hamiltonian property}, or \emph{PH-property}, if every pairing $M$
of $G$ admits a perfect matching $N\subseteq E(G)$, disjoint from $M$, such
that $M\cup N$ is a Hamiltonian cycle of $K_G$. We also say that $G$ is
\emph{PH}, and call $N$ a \emph{Hamiltonian completion} of $M$ in $G$.
Note that the edges of $M$ need not belong to $G$. With this terminology, later
introduced by Alahmadi et al. in \ \cite{Alahmadi2015}, Fink's result can be stated
as follows.

\begin{theorem}[Fink {\cite{Fink2007}}, 2007]\label{thm:fink}
The hypercube $Q_d$ has the PH-property for every $d\geq2$.
\end{theorem}

Recall that the Cartesian product of two graphs $G$ and $H$, denoted by
$G\square H$, has vertex set $V(G)\times V(H)$, where $(u,x)$ and $(v,y)$ are
adjacent precisely when $u=v$ and $xy\in E(H)$, or $x=y$ and $uv\in E(G)$.
For each $u\in V(G)$, the set $\{u\}\times V(H)$ induces a copy of $H$, called
an \emph{$H$-fibre}. The prism $G\square K_2$ consists of two copies of $G$
joined at corresponding vertices.

For a graph $G$, we write $\omega(G)$ for its number of connected components.
For a matching $M$ and a set $S \subseteq V(G)$ of vertices, let $M[S]$ denote the pairs of
$M$ with both ends in $S$. In a Cartesian product, a perfect matching $F$ of
a graph $G$ is identified with the spanning graph $(V(G),F)$.
In intermediate arguments, a common edge of two perfect matchings is
represented by two labelled copies. Their union is then a multigraph whose
components are even alternating cycles, including cycles of length two.
All final Hamiltonian cycles are ordinary unions of disjoint matchings.
We also use the well-known fact that a graph containing a spanning PH subgraph is itself PH.

The Cartesian-product viewpoint provides a natural setting for extending
Fink's theorem. Since $Q_{d+1}\cong Q_d\square K_2$, hypercubes arise by
iterating the prism operation starting from $Q_2=C_4$. Results concerning
the PH-property and related matching-extension properties under graph
operations can be found in \cite{AGZ2025,AMRZ2025,Alahmadi2015,RomanielloZerafa2023}.

In particular, the authors of \ \cite{AMRZ2025} proved that the
prism of every PH graph is PH, see Theorem \ref{lem:prism}. In the same paper,
they asked whether the Cartesian product of any two PH graphs is PH.
Our main result, Theorem~\ref{thm:main}, answers this question affirmatively.

The proof combines a composition principle with a parity observation
concerning the two projections of a pairing. Path compression appears in
Fink's proof \cite{Fink2007}, and Fon-Der-Flaass
\cite[Lemma~1(i)]{FonDerFlaass2010} formulates the corresponding gluing
principle for two blocks. For disjoint subcubes, Gregor \cite{Gregor2009}
gives a connectedness criterion for completions contained in those subcubes;
see also Fink \cite{Fink2020}. We use a formulation for arbitrary PH blocks
and show how to choose, for each pairing of a Cartesian product, a suitable
partition into prisms.

More precisely, for each pairing of $G\square H$, our construction produces
a completion contained in prisms determined by a perfect matching of $G$
or of $H$. Although the PH-property of $G\square H$ alone does not imply that
both $G$ and $H$ are PH, this support restriction yields an exact
characterization when it is required for every pairing; see Theorem~\ref{thm:prismatic-iff}.

The paper is organized as follows: Section~\ref{sec:preliminaries} contains the preparatory lemmas, including
the published prism theorem, while Section~\ref{sec:mainresults} proves Cartesian
closure and the characterization by prismatic support. Section~\ref{sec:concluding} concludes the manuscript with some final remarks.

\section{Preparatory lemmas}\label{sec:preliminaries}
We collect the results needed to prove the main theorem, beginning with a
composition lemma for vertex-disjoint PH graphs.

\begin{lemma}\label{lem:blocks}
Let $t\geq1$, and let $X_1,\ldots,X_t$ be simple PH graphs of even order
at least four, with pairwise disjoint vertex sets $V_1,\ldots,V_t$. Put $V=V_1\dot{\cup}\cdots\dot{\cup} V_t$, let $X$ be the disjoint union of the PH graphs $X_1,\ldots,X_t$,
where $t\geq1$, and write $V_i=V(X_i)$. Let $M$ be any pairing
of $X$. Define the simple \emph{incidence graph} $\Lambda_M$ with $ V(\Lambda_M)=\{1,\ldots,t\}$ by joining
$i\neq j$ if a pair of $M$ has one end in $V_i$ and the other in $V_j$;
isolated vertices are retained.

Then there exist perfect matchings $N_i\subseteq E(X_i)$ such that, with
$N=\bigcup_{i=1}^{t}N_i$, we have $M\cap N=\varnothing$ and
\begin{linenomath*}
\[
 \omega\bigl((V(X),M\cup N)\bigr)=\omega(\Lambda_M).
\]
\end{linenomath*}
Consequently, $M$ admits a Hamiltonian completion contained in
$\bigcup_{i=1}^{t}E(X_i)$ if and only if $\Lambda_M$ is connected.
\end{lemma}
\begin{proof}
Each $X_i$ has a perfect matching. Among all choices $N=\bigcup_iN_i$, with $N_i$
a perfect matching of $X_i$, choose one minimizing the number of alternating
components of $M\cup N$. At this stage, $N$ is not required to be
disjoint from $M$.

Suppose that $V_i$ meets $q\geq2$ components. Delete all edges of $N_i$. Every
vertex of $V_i$ now has degree one, whereas all other vertices have degree two.
Every cycle meeting $V_i$ contained an edge of $N_i$, since the partner in $N$
of a vertex of $V_i$ lies in the same block. The affected cycles therefore
split into vertex-disjoint paths whose endvertices are exactly the vertices
of $V_i$. No vertex of $V_i$ is internal to one of these paths.

Record the endpoint pairs of these paths as a pairing $R_i$ of $V_i$. Since
$X_i$ is PH, there exists a perfect matching $N_i'$ of $X_i$ such that
$R_i\cup N_i'$ is a Hamiltonian cycle on $V_i$. Replacing $N_i$ by $N_i'$ and
expanding the compressed paths merges the $q$ affected cycles into one, while
leaving every other cycle unchanged. This decreases the number of components
by $q-1$, a contradiction.

Consequently, each block is entirely contained in one cycle. An edge of $M$
between two blocks forces them to lie in the same cycle. Thus all blocks in a
component of $\Lambda_M$ belong to one cycle, whereas neither $M$ nor $N$ joins
different components of $\Lambda_M$. The stated equality follows. Every final
cycle contains a whole block of at least four vertices, so no component has
length two and $M\cap N=\varnothing$.
\end{proof}

The two-block case is the gluing principle of
\cite[Lemma~1(i)]{FonDerFlaass2010}. For subcubes of dimension at least two,
the connectedness condition is Gregor's criterion \cite{Gregor2009}, also
presented in \cite{Fink2020}. The preceding lemma does not require an ambient
hypercube or a product structure within the blocks. The convention of
\cite{FonDerFlaass2010} also admits $K_2$ as a degenerate \emph{Fink graph};
here PH blocks have order at least four.

The next lemma does not require either the graphs $G$ or $H$ to be PH.

\begin{lemma}\label{lem:parity}
Let $G,H$ be simple graphs such that $|V(G)|\,|V(H)|$ is even, and let
$M$ be any pairing of $V(G)\times V(H)$. Define simple graphs $\Gamma_G(M)$
and $\Gamma_H(M)$ on $V(G)$ and $V(H)$, respectively, by
\begin{linenomath*}
\begin{align*}
 uv\in E(\Gamma_G(M))
 &\iff u\neq v\text{ and }(u,x)(v,y)\in M\text{ for some }x,y\in V(H),\\
 xy\in E(\Gamma_H(M))
 &\iff x\neq y\text{ and }(u,x)(v,y)\in M\text{ for some }u,v\in V(G).
\end{align*}
\end{linenomath*}
All vertices, including isolated ones, are retained; the projection graphs
need not be subgraphs of $G$ and $H$. Then every component of $\Gamma_G(M)$
has even order, or every component of $\Gamma_H(M)$ has even order
(possibly both).
\end{lemma}
\begin{proof}
Let $A_1,\ldots,A_s$ and $B_1,\ldots,B_t$ be the vertex sets of the components
of the two projections. If $(u,x)\in A_i\times B_j$ is paired with $(v,y)$,
then $v\in A_i$ and $y\in B_j$. Indeed, a change of coordinate gives an edge
of the corresponding projection, and an unchanged coordinate remains in its
component. Hence each rectangle $A_i\times B_j$ is invariant under $M$ and is
covered by its restriction. Therefore
\begin{linenomath*}
\[
 |A_i|\,|B_j|\equiv0\pmod2\qquad\text{for all }i,j.
\]
\end{linenomath*}
If every $A_i$ has even order, the conclusion holds. Otherwise one $A_i$ has
odd order, forcing every $B_j$ to have even order.
\end{proof}

\begin{lemma}\label{lem:connect}
Let $G$ be a simple PH graph of even order at least four. Let $R$ be
any simple graph with $V(R)=V(G)$ such that every connected component
of $R$ has even order; no inclusion $E(R)\subseteq E(G)$ is assumed. Then
there exists a perfect matching $F\subseteq E(G)$ such that the graph
$(V(G),E(R)\cup F)$ is connected. The perfect matching $F$ is not required to be
disjoint from $E(R)$.
\end{lemma}
\begin{proof}
Choose a pairing $L$ of $V(G)$ whose pairs lie within components of $R$. This
is possible because their orders are even; neither $L\subseteq E(R)$ nor
$L\subseteq E(G)$ is required. Since $G$ is PH, there is a perfect matching
$F$ of $G$ for which $L\cup F$ is Hamiltonian. If $(V(G),E(R)\cup F)$ were disconnected,
every pair of $L$ would remain in one of its components, since its ends lie
in the same component of $R$. Then $L\cup F$ would also be disconnected,
a contradiction.
\end{proof}

The last tool needed is the following theorem from \cite{AMRZ2025}, which we restate as:

\begin{thm}[Abreu et al.\ {\cite[Theorem~2.1]{AMRZ2025}}]\label{lem:prism}
Let $G$ be a graph having the PH-property. Then, $G \square K_2$ is PH.
\end{thm}

\section{Cartesian closure and prismatic support}\label{sec:mainresults}
We first prove Cartesian closure by choosing a partition into PH prisms
adapted to the assigned pairing.

\begin{theorem}\label{thm:main}
If $G$ and $H$ have the PH-property, then $G\square H$ has the PH-property.
\end{theorem}

\begin{proof}
Let $M$ be any pairing of $G\square H$. By Lemma~\ref{lem:parity}, one of its
projections has all components of even order. Interchanging $G$ and $H$ if
necessary, suppose this is $\Gamma_G(M)$. Lemma~\ref{lem:connect} provides a
perfect matching $F$ of $G$ for which $\Gamma_G(M)\cup F$ is connected.
For every $uv\in F$, set
\begin{linenomath*}
\[
 W_{uv}=\{u,v\}\times V(H).
\]
\end{linenomath*}
Since $F$ is a perfect matching, these sets partition $V(G\square H)$. Each one induces a prism
$K_2\square H$, which is PH by Theorem~\ref{lem:prism}.

The incidence graph of $M$ on these blocks is obtained from
$\Gamma_G(M)\cup F$ by contracting the edges of $F$ and discarding loops and
parallel edges. Indeed, two contracted pairs are adjacent precisely when an
edge of $M$ joins the corresponding blocks. The incidence graph is therefore
connected. Lemma~\ref{lem:blocks} supplies a perfect matching $N$ contained in
the prisms such that $M\cup N$ is Hamiltonian. All prism edges belong to
$G\square H$, so $N\subseteq E(G\square H)$. Since $M$ was arbitrary, the product is PH.
\end{proof}

The result extends immediately to Cartesian products of finitely many PH graphs.

\begin{corollary}\label{cor:finite}
If $G_1,\ldots,G_t$ are PH, where $t\geq1$, then $G_1\square\cdots\square G_t$ is PH.
\end{corollary}
\begin{proof}
Apply Theorem~\ref{thm:main} inductively on $t$.
\end{proof}

The PH-property of $G\square H$ alone does not imply that both $G$ and $H$
are PH. For instance, $C_6\square K_6$ is PH by
\cite[Theorem~2]{Alahmadi2015}, whereas $C_6$ is not PH. Moreover, let $G=C_{10}\square C_4$ and $H=K_{1,10}\square Q_3$.
The graph $G$ is not PH \cite{GauciZerafa2021}, whereas $H$
is non-Hamiltonian by \cite[Theorem~2]{Dimakopoulos2005};
hence neither graph is PH. Nevertheless, $G\square H$ is PH,
as follows by combining \cite[Corollary~1]{Dimakopoulos2005}
with \cite[Corollary~3.1]{AMRZ2025}.
A converse holds under the support restriction furnished by the proof of
Theorem~\ref{thm:main}, which we now formalize.

\begin{definition}\label{def:prismatic}
Let $G,H$ be graphs of even order at least four, and let $M$ be a pairing of
$V(G)\times V(H)$. A Hamiltonian completion $N$ of $M$ in $G\square H$ has
\emph{prismatic support} if at least one of the following holds:
\begin{linenomath*}
\begin{align}
N&\subseteq E(F_G\square H)
&&\text{for some perfect matching }F_G\subseteq E(G),\label{eq:support-G}\\
N&\subseteq E(G\square F_H)
&&\text{for some perfect matching }F_H\subseteq E(H).\label{eq:support-H}
\end{align}
\end{linenomath*}
The choice of $G$ or $H$ and the corresponding perfect matching may depend
on $M$. This definition does not assume that $G$, $H$, or the resulting
prisms are PH.
\end{definition}

\begin{theorem}\label{thm:prismatic-iff}
Let $G,H$ be simple graphs, each of even order at least four. Then
$G$ and $H$ are PH if and only if every pairing of $G\square H$ has a Hamiltonian
completion with prismatic support.
\end{theorem}
\begin{proof}
Suppose first that $G$ and $H$ are PH. For any pairing of $G\square H$,
the proof of Theorem~\ref{thm:main} constructs a Hamiltonian completion
contained in the disjoint prisms determined by a perfect matching of $G$
or of $H$. Thus the completion has prismatic support.

Conversely, let $P$ be any pairing of $V(G)$ and consider the pairing
\begin{linenomath*}
\[
 M_P=\{(u,h)(v,h):uv\in P,\ h\in V(H)\}
\]
\end{linenomath*}
of the product. Every edge of $M_P$ preserves the second coordinate. By
assumption, it has a Hamiltonian completion $N$ with prismatic support.

Condition~\eqref{eq:support-H} is impossible. If
$N\subseteq E(G\square F_H)$, each set $V(G)\times\{h,h'\}$, with $hh'\in F_H$,
would be invariant under both $M_P$ and $N$. There are $|V(H)|/2\geq2$ such
nonempty sets, so $M_P\cup N$ would be disconnected.

Consequently, there is a perfect matching $F_G$ of $G$ satisfying
\eqref{eq:support-G}. If $P\cup F_G$ were disconnected, let $S$ be the vertex
set of a proper component. Then $S\times V(H)$ would be invariant under
$M_P$ and under $N$: an edge of $F_G\square H$ either preserves the first
coordinate or changes it along an edge of $F_G$. This again contradicts
connectedness of $M_P\cup N$.

Thus $P\cup F_G$ is connected on at least four vertices. A common edge of
$P,F_G$ would form an isolated component on its two ends. Hence the matchings
are disjoint, and their connected, 2-regular union is a Hamiltonian cycle.
Since $P$ was arbitrary, $G$ is PH. Interchanging $G$ and $H$ proves that
$H$ is PH as well.
\end{proof}

\section{Concluding remarks}\label{sec:concluding}

Theorem~\ref{thm:main} establishes Cartesian closure of the
PH-property, while Theorem~\ref{thm:prismatic-iff} identifies
a support condition under which a converse holds. The distinction
is essential: the existence of Hamiltonian completions in a
Cartesian product does not, by itself, imply the PH-property
of either of the two graphs.

The proof also gives a criterion for a prescribed prismatic
support. Let $H$ be PH, let $G$ be a graph of even order at least
four, and fix a perfect matching $F\subseteq E(G)$.
For any pairing $M$ of $G\square H$, there is a Hamiltonian
completion of $M$ contained in $E(F\square H)$ if and only if
the graph
\begin{linenomath*}
\[
  \bigl(V(G),E(\Gamma_G(M))\cup F\bigr)
\]
\end{linenomath*}
is connected. Indeed, contracting the edges of $F$ gives the
incidence graph of $M$ on the corresponding PH prisms, and
Lemma~\ref{lem:blocks} applies. No PH assumption on $G$ is
needed for this criterion.

Without any restriction on the support of the completion,
one may instead ask whether preservation of the PH-property
with every PH graph is sufficient to recover the property
of the original graph.

\begin{quest}\label{quest:universal}
Let $G$ be a finite simple graph of even order at least four.
If $G\square H$ is PH for every PH graph $H$, must $G$ be PH?
\end{quest}

The examples in Section~\ref{sec:mainresults} concern particular
products and do not answer this question. Nor does
Theorem~\ref{thm:prismatic-iff}, since the hypothesis above
does not require the completions to have prismatic support.

Another direction concerns the existence of PH graphs with
large girth. Alahmadi et al.\
\cite[Open Problem~4]{Alahmadi2015} asked whether there are
infinitely many PH graphs of girth at least five, noting that
no such graph was known to them. A first question is therefore
whether even a single example exists.

\begin{quest}\label{quest:girth}
Does there exist a finite simple PH graph of girth at least five?
\end{quest}

The Cartesian-product construction 
cannot directly provide such an example. Indeed, whenever
$uv\in E(G)$ and $xy\in E(H)$, the vertices
\begin{linenomath*}
\[
  (u,x),\ (v,x),\ (v,y),\ (u,y)
\]
\end{linenomath*}
form a cycle of length four in $G\square H$. Consequently,
every Cartesian product of two PH graphs has girth at most
four. Constructing PH graphs of girth at least five therefore
requires an approach beyond direct applications of Cartesian
closure.

\end{document}